\documentclass[intlimits]{article}
\usepackage{latexsym,amsfonts,amssymb,amsmath,amsthm}
\date{}

\usepackage{graphicx}
\usepackage{float}
\usepackage{caption}

\usepackage{pgfplots}
\usepackage{tikz}
\usetikzlibrary{matrix}
\usetikzlibrary{decorations.markings}
\usetikzlibrary{calc}
\usetikzlibrary{shapes}
\usetikzlibrary{arrows.meta, bending}
\usetikzlibrary{backgrounds}
\pgfplotsset{compat=1.14}

\usepackage{mathtools}

\newtheorem{theorem}{Theorem}[section]
\newtheorem{lemma}[theorem]{Lemma}

\newtheorem{proposition}[theorem]{Proposition}
\newtheorem*{thm:main}{Main Theorem}

\theoremstyle{definition}
\newtheorem{definition}[theorem]{Definition}

\theoremstyle{remark}
\newtheorem{remark}{Remark}

\newcommand{\NN}{\ensuremath{\mathbb{N}}}
\newcommand{\RR}{\ensuremath{\mathbb{R}}}

\newcommand{\lambdaext}{\ensuremath{\lambda_{\mathrm{ext}}}}
\newcommand{\lambdaprinc}{\ensuremath{\lambda_{\mathrm{princ}}}}

\newcommand{\norm}[1]{\ensuremath{\lVert#1\rVert}}

\DeclareMathOperator{\card}{card}

\DeclareMathOperator{\conv}{conv}
\DeclareMathOperator{\dist}{dist}

\DeclareMathOperator{\Inte}{Int}

\DeclareMathOperator{\spanned}{span}

\begin{document}

\title{The $C^1$ property of convex carrying simplices for competitive maps}

\author{
Janusz Mierczy\'nski \\
Faculty of Pure and Applied Mathematics \\
Wroc{\l}aw University of Science and Technology \\
Wybrze\.ze Wyspia\'nskiego 27 \\
PL-50-370 Wroc{\l}aw \\
Poland}

\maketitle

\begin{abstract}
For a class of competitive maps there is an invariant one-codimensional manifold (the carrying simplex) attracting all non-trivial orbits.  In the present paper it is shown that its convexity implies that it is a $C^1$ submanifold-with-corners, neatly embedded in the non-negative orthant.  The proof uses the characterization of neat embedding in terms of inequalities between Lyapunov exponents for ergodic invariant measures supported on the boundary of the carrying simplex.
\end{abstract}

\section{Introduction}
\label{sect:intro}
In his paper \cite{H-1988} M. W. Hirsch proved, for a wide class of totally competitive systems of ordinary differential equations (ODEs), the existence of an unordered (with respect to the coordinate-wise ordering) invariant set, homeomorphic to the standard probability simplex via radial projection, such that every non-trivial (that is, not equal identically to zero) orbit is attracted towards it.  Some gaps in Hirsch's proof in~\cite{H-1988}, noticed by M. L. Zeeman, were filled in later proofs of discrete-time analogues, see below.

M. L. Zeeman in~\cite{Z-1993} introduced the name `carrying simplex'.  It would be good to give some explanation of the meaning of the name.  Namely, totally competitive systems model interactions between species where both interspecific and intraspecific competition are present.  The simplest ODE of that kind, the logistic equation $\dot{x} = r x (1 - \frac{x}{K})$, where $r, K > 0$, has the property that its positive solutions tend, as $t \to \infty$, to $K$.  The biological interpretation is that the population density eventually stabilizes at $K$ (which is called the {\em carrying capacity\/}).  In the multidimensional case, as each non-trivial solution tends to that simplex, the eventual behaviour is given by the restriction of the system on the simplex.  We mention here that, as shown by Smale~\cite{Sma}, every smooth vector field on the standard $n$-dimensional probability simplex can be embedded in a totally competitive system of dimension $n + 1$.

Hirsch's theory of carrying simplices has been extended to some discrete-time competitive systems.  To the present author's knowledge, the first paper to deal with those issues is H. L. Smith~\cite{Sm}.  There, a class of competitive maps was introduced which contained Poincar\'e maps of time-periodic totally competitive systems of ODEs, and for that class the existence of a carrying simplex was proved.  Some problems left open in~\cite{Sm} were resolved in Wang and Jiang~\cite{W-J-2002}.  In~\cite{H-2008} Hirsch announced a theorem on the existence of a carrying simplex for some discrete-time competitive systems (not necessarily invertible).  In~\cite[Appendix]{RH} Ruiz-Herrera proved the existence of the carrying simplex for maps that are not necessarily bijective.  A similar result was obtained by Jiang, Niu and Wang in~\cite[Appendix]{J-N-W}.  For a different, more `dynamical' in spirit, proof, see Baigent~\cite{B-JDDE} (analogues for Lotka--Volterra systems of ODEs were given in \cite{B-Edinburgh}, \cite{B}).

Among other results, Hirsch showed in~\cite{H-1988} that the carrying simplex is a Lipschitz submanifold.  He asked there whether the simplex is smooth.

The solution of the problem of the smoothness of the carrying simplex inside the non-negative orthant has been eluding, and continues to elude, researchers so far.  Indeed, to the present author's best knowledge, no counterexamples are known.

A consequence of the above is that, although we have reduction to a system of smaller dimension, that reduction is at most Lipschitz.  Apparently, this should have no relevance when one uses purely topological tools such as CW decompositions (see~\cite{B-E-L}) or degree (\cite{J-N}, and recently, \cite{N-RH}), but even then the (possible) absence of smoothness can be a serious problem.

\smallskip
The picture changes when one takes into account the boundary of the carrying simplex.  Namely, the most general result is that the carrying simplex is a $C^1$ submanifold-with-corners neatly embedded in the non-negative orthant if and only if for any ergodic invariant measure supported on the boundary of the simplex the external Lyapunov exponents are larger than the principal Lyapunov exponent \cite{J-M-W} (for an earlier continuous-time version, with the 'if' implication only, see~\cite{M-1994}).

In~particular, a sufficient condition for the satisfaction of the above property of invariant measures is that for each face of the carrying simplex  the boundary of the face be weakly repelling (that property is called, in the context of mathematical ecology, \emph{weak persistence}). See~\cite{M-1994}, and for the strong repelling (\emph{permanence}), \cite{M-Sch}.

\medskip
Convexity of the carrying simplex (or, more precisely, convexity of the global attractor) is a feature that has been intensively investigated recently (see, for~example, \cite{B, B-JDDE}).  As, at~least for discrete-time competitive maps being analogues of Lotka--Volterra competitive systems of ODEs, there is strong correlation between the convexity of the carrying simplex and the global asymptotic stability of the (necessarily unique) interior fixed point (\cite{B}; cf.~also \cite{T, Z, Z-Z}), it is not unnatural, in view of the results mentioned in the last paragraphs, to suspect that the convexity of the carrying simplex should imply the appropriate inequalities on the Lyapunov exponents.  The present paper shows that this is indeed the case for a general class of competitive maps.

The result for three-dimensional systems has been proved in~\cite{M-2017}.

\smallskip
The paper is organized as follows.  In Section~\ref{sect:preliminaries} the notation is introduced, the necessary definitions and results are given. The Main Theorem is formulated in Section~\ref{sect:main}. Section~\ref{sect:main-proof} is devoted to the proof of the main result.

\section{Carrying simplices: Existence and basic properties}
\label{sect:preliminaries}

For a metric space $W$ let $\mathcal{B}(W)$ stand for the $\sigma$-algebra of Borel subsets of $W$.  $\NN$ denotes the set of positive integers.

\smallskip
We shall distinguish between points (elements of the {\em affine\/} space $H = \{\, x = (x_1, \dots, x_N): x_i \in \RR \,\}$) and vectors (elements of the {\em vector\/} space $V = \{\, v = (v_1, \dots, v_N): v_i \in \RR \,\}$).  $\norm{\cdot}$ stands for the Euclidean norm in $V$.
Denote by $C = \{\,x \in H: x_i \ge 0$ for all $i = 1, \dots, N \,\}$ the non-negative {\em orthant\/}.

\smallskip
The interior of $C$ is $C^{\circ} := \{\,x \in H : x \gg 0\,\}$ and the boundary of $C$ is $\partial C = C \setminus C^{\circ}$.

For $I \subset \{1,\dots,N\}$, let $C_I := \{\,x \in C_I : x_j = 0$ for $j \in \notin I \,\}$, $\dot{C}_I := \{\,x \in C_I : x_i > 0$ for $i \in I \,\}$ and $\partial C_I$ be the relative boundary of $C_I$, $\partial C_I = C_I \setminus \dot{C}_I$.  $C_I$ is called a {\em $k$-dimensional face\/} of $C$, where $k = \card{I}$.

For $x, y \in C_I$, we write $x \leq_I y$ if $x_i \leq y_i$ for all $i \in I$, and $x \ll_I y$ if $x_i < y_i$ for all $i \in I$.  If $x \leq_I y$ but $x \neq y$ we write $x <_I y$ (the subscript in $\leq$, $<$, $\ll$ is dropped if $I = \{1,\dots,N\}$).

\smallskip
The standard non-negative {\em cone\/} $K$, with nonempty interior $K^{\circ}$, in $V$ is the set of all $v$ in $V$ such that $v_i \ge 0$ for all $i \in \{1,\dots,N\}$.  For $I \subset \{1,\dots,N\}$, let $V_I := \{\,v \in V: v_j = 0$ for $j \notin I$ $\,\}$. The relations $\leq_I$, $<_I$, $\ll_I$ between vectors in $V_I$ are defined as the relations between points in $C_I$.  Let $K_I := K \cap V_I$, $\dot{K}_I := \{\,v \in K_I: v_i > 0$ for $i \in I\,\}$, $\partial K_{I} := K_I \setminus \dot{K}_I$.

\medskip
Let $P \colon C \to P(C)$ be a $C^k$ diffeomorphism onto its image $P(C) \subset H$.  Recall that this means that there is an open $U \subset H$, $C \subset U$, and a $C^k$ diffeomorphism $\tilde{P} \colon U \to \tilde{P}(C)$ such that the restriction $\tilde{P}|_{C}$ of $\tilde{P}$ to $C$ equals $P$.

\smallskip
A set $A \subset C$ is {\em invariant\/} if $P(A) = A$.  For $x \in C$ the {\em $\omega$-limit set $\omega(x)$\/} of $x$ is the set of those $y \in C$ for which there exists a subsequence $n_k \to \infty$ such that $\norm{P^{n_m}x - y} \to 0$ as $m \to \infty$.  A compact invariant $\Gamma \subset C$ is the {\em global attractor\/} for $P$ if for each bounded $B \subset C$ and each $\epsilon > 0$ there exists $n_0 \in \NN$ such that $P^{n}(B)$ is contained in the $\epsilon$-neighbourhood of $\Gamma$
for $n \ge n_0$.

\medskip
As in~\cite{J-M-W} we introduce the following assumptions.  We assume throughout the paper that they are satisfied.

\vskip 2mm
\noindent\textbf{(H1)} {\em $P$ is a $C^2$ diffeomorphism onto its
image $P(C)$.}

\vskip 2mm
\noindent \textbf{(H2)} {\em For each nonempty $I \subset \{1,\dots,N\}$, the sets $A = C_I, \dot{C}_I$ and $\partial C_I$ have the property that $P(A) \subset A$ and $P^{-1}(A) \subset A$.}

\vskip 2mm
\noindent \textbf{(H3$'$)} {\em For each nonempty $I \subset \{1,\dots,N\}$ and $x \in \dot{C}_I$, the $I \times I$  Jacobian matrix $D(P|_{C_I})(x)^{-1} = (DP(x)^{-1})_I = (DP^{-1}(Px))_I$ has all entries positive. Moreover, for any non-zero $v \in K_{\{1, \ldots, N\}} \setminus I$ there exists $j \in I$ such that $(DP(x)^{-1}v)_j > 0$.}

\vskip 2mm
\noindent \textbf{(H4$'$)} {\em For each $i \in \{1,\dots,N\}$, $P|_{C_{\{i\}}}$ has a unique fixed point $u_i > 0$ with $0 < (d/dx_{i})(P|_{C_{\{i\}}})(u_i) < 1$.  Moreover, $\dfrac{\partial P_i}{\partial x_j}(u_i) < 0$ $(j \ne i)$.}

\vskip 2mm
\noindent \textbf{(H5)} \textit{If $x \in \dot{C}_I$ is a nontrivial $p$-periodic point of $P$, then $\mu_{I,p}(x) < 1$, where $\mu_{I,p}(x)$ is the \textup{(}necessarily real\textup{)} eigenvalue of the mapping $D{(P|_{C _I})}^p(x)$ with the smallest modulus.}

\vskip 2mm
\noindent \textbf{(H6)}  {\em For each nonempty $I \subset \{1,\dots,N\}$ and $x, y \in \dot{C}_I$, if $0 \ll_{I} Px \ll_{I} Py$, then $\dfrac{P_{i}x}{P_{i}y} \ge \dfrac{x_i}{y_i}$ for all $i \in
I$ \textup{(}where $P = (P_1, \dots, P_N)$\textup{)}.}

\vskip 2mm
Maps satisfying (H1), (H2) and some weaker form of (H3$'$) (without the last sentence) are called in~\cite{Sm} {\em competitive maps\/}.

\begin{remark}
  \label{rem-1}
  Observe that the restriction of the mapping satisfying (H1) to (H6) to any face $C_I$ of nonzero dimension satisfies all the above assumptions, modulo relabelling.
\end{remark}

\smallskip
The symbol $\triangle$ stands for the {\em standard probability
$(N - 1)$-simplex}, $\triangle : =\{\,x \in C: \sum_{i = 1}^N x_i = 1\,\}$.

\begin{theorem}
\label{thm0}
There exists a compact invariant $S \subset C$ \textup{(}the {\em carrying simplex} for $P$\textup{)} having the following properties:
\begin{enumerate}
\item[{\rm (i)}]
$S$ is homeomorphic to the standard probability simplex $\triangle$ via radial projection $R$.
\item[{\rm (ii)}]
No two points in $S$ are related by the $\ll$ relation.  Moreover, for each nonempty $I \subset \{1,\dots,N\}$ no two points in $S \cap \dot{C}_I$ are related by the $<_I$ relation.
\item[{\rm (iii)}]
For any $x \in C \setminus \{0\}$ one has $\omega(x) \subset S$.
\item[{\rm (iv)}]
The global attractor $\Gamma$ equals $\{\, {\alpha} x : \alpha \in [0, 1], \ x \in S \,\}$.
\end{enumerate}
\end{theorem}
\begin{proof}
Parts (i), (iii) and the first sentence in part (ii) correspond to parts of~\cite[Thm.~0]{J-M-W}.  The second sentence in part (ii) follows from an application of~\cite[Thm.~0(ii)]{J-M-W} to the carrying simplex of the restriction of the map $P$ to $C_I$.  Part (iv) is \cite[Prop.~2.4]{J-M-W}.
\end{proof}

We let $S_I := S \cap C_I$, $\dot{S}_I = S \cap \dot{C}_I$, $\partial S_I := S \cap \partial C_I$ and $S^{\circ} := S \cap C^{\circ}$.  A set $S_I$ is called a {\em $k$-dimensional face\/} of $S$, $k = \card{I}-1$.

\section{Statement of the main results}
\label{sect:main}

\bigskip
Before stating our main results we shortly explain what is meant by the neat embedding of the carrying simplex.  The differential-topological formulation of that concept is rather complicated, and the reader is referred to \cite{J-M-W} (in the present context).

For our purposes in the present paper the following suffices.  $S$ is a {\em $C^1$ submanifold-with-corners neatly embedded in $C$\/} when $R^{-1}|_{\triangle}$ is a $C^1$ mapping (recall that $R|_S$, the radial projection $R$ restricted to $S$, is a homeomorphism between $S$ and the standard probability simplex $\triangle$).  In other words, $S$ is a $C^1$ submanifold-with-corners neatly embedded in $C$ if~and only~if there exists a $C^1$ function $\rho \colon \triangle \to (0, \infty)$ such that
\begin{equation*}
(R|_S)^{-1}(y) = \rho(y) y, \quad y \in \triangle.
\end{equation*}
One can imagine neat embedding as `as little tangency as possible.'  In~general, the position of the carrying simplex can be quite complicated: a part of the face of the carrying simplex can be tangent to the corresponding face of the orthant, whereas at another part we have transversality.  See~\cite{M-1999a} and the figure below.
\begin{figure}[H]
  \centering
  \includegraphics[width=0.8\textwidth]{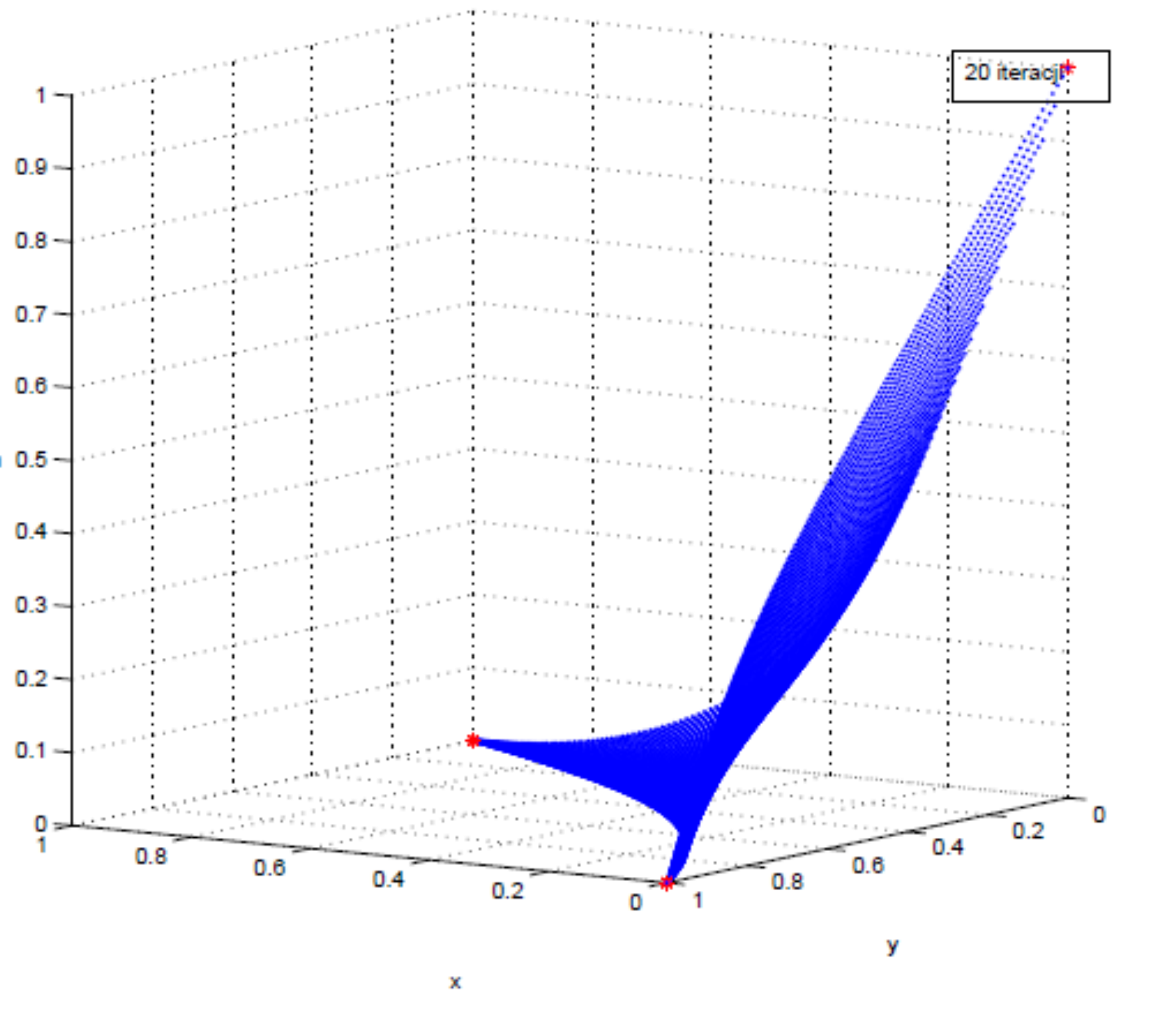}
  \caption{\small This is a picture (taken from~\cite{Mar}) of the carrying simplex that is NOT neatly embedded.  The ODE system is $\dot{x}_1 = x_1 (1 - x_1 - 1.9\, x_2 - 0.2\, x_3)$, $\dot{x}_2 = x_2 (1 - 0.9\; x_1 - x_2 - 0.5\, x_3)$, $\dot{x_3} = x_3 (1 - 4\, x_1 - 1.1\, x_2 - x_3)$.  $S$ is tangent to $C_{\{1,2\}}$ along $S_{\{1,2\}} \setminus \{(0,1,0)\}$, whereas the tangent cone at $(0, 1, 0)$ contains a vector with third coordinate positive.}\label{fig:1}
\end{figure}

\begin{definition}
\label{def:convex}
The carrying simplex $S$ is {\em convex\/} if the set $\{\, {\alpha} x : \alpha \in [0, 1], \ x \in S \,\}$, that is, the global attractor $\Gamma$, is convex.
\end{definition}

We are ready now to formulate the main theorem in the paper.
\begin{thm:main}
Assume that $S$ is convex.  Then $S$ is a $C^1$ submanifold-with-corners neatly embedded in $C$.
\end{thm:main}

In the proof of Main Theorem we will make heavy use of the characterization of the neat embedding of $S$ given in~\cite{J-M-W}.  We introduce now the required concepts.

Let $DP$ denote the linear skew-product dynamical system,
\begin{equation*}
DP(x, v) = (Px, DP(x)v), \quad x \in C, \ v \in V,
\end{equation*}
induced by $P$ on the product bundle $C \times V$.  The restriction of $DP$ to $\Gamma \times V$ is a bundle automorphism.

Fix an invariant ergodic measure $\mu$ supported on the boundary $\partial S$. Let $I(\mu) \subset \{1, \dots, N\}$ be such that $\mu(\dot{S}_{I(\mu)}) = 1$.  By ergodicity, such an $I(\mu)$ is unique.  By (H2), the product bundles $S_{I(\mu)} \times V_{I(\mu)}$ and $S_{I(\mu)} \times V$ are invariant under $DP$.  The Lyapunov exponents for $DP|_{S_{I(\mu)} \times V_{I(\mu)}}$ (for definitions, see, e.g., \cite[Thm.~3.4.11]{Arn}) are called {\em internal Lyapunov exponents\/}.  The smallest internal Lyapunov exponent is called the {\em principal Lyapunov exponent\/}.  For each $j \in \{1, \dots, N\} \setminus I(\mu)$ the bundle $S_{I(\mu)} \times V_{I(\mu) \cup \{j\}}$ is invariant and contains $S_{I(\mu)} \times V_{I(\mu)}$.  The `additional'  Lyapunov exponent for $DP|_{S_{I(\mu)} \times V_{I(\mu) \cup \{j\}}}$ is called the $j$-th {\em external Lyapunov exponent\/}.

For the following, see \cite[Thm.~A]{J-M-W}.
\begin{theorem}
\label{thm:equivalence}
$S$ is a $C^1$ submanifold-with-corners neatly embedded in $C$ if and only if for each ergodic invariant measure $\mu$ supported on $\partial S$ the principal Lyapunov exponent is smaller than all external Lyapunov exponents.
\end{theorem}

\section{Proof of the Main Theorem}
\label{sect:main-proof}

Our aim in this section is the proof of the following result.
\begin{theorem}
  \label{thm:main-alt}
  Assume that $S$ is convex.  Then for each ergodic invariant measure $\mu$ supported on $\partial S$ the principal Lyapunov exponent is smaller than all external Lyapunov exponents.
\end{theorem}
In view of the characterization given in Theorem~\ref{thm:equivalence} this will be equivalent to proving the Main Theorem.

\medskip
The proof goes by induction on the cardinality of $I(\mu)$.

\smallskip
Let, for $1 \le k \le N - 1$, (H$_k$) denote the following statement:

    \noindent {\em For any ergodic invariant measure $\mu$ with $\card{I(\mu)} \le k$ the principal Lyapunov exponent is smaller than the external Lyapunov exponents.}

\smallskip
A main role in the inductive procedure employed in~\cite{J-M-W} in the proof of the `if' part  of Theorem~\ref{thm:equivalence} (in our notation)  was played by the following result (cf.~Fundamental Induction Hypothesis on p.~1640 of~\cite{J-M-W}).  It will play a similar role in the proof of the Main Theorem here.

\begin{proposition}
\label{prop:induction}
    Assume \textup{(H$_k$)}.  Then each $S_I$ with $\card{I} = k$ is a $C^1$ submanifold-with-corners neatly embedded in $C_I$.  Moreover, there is an invariant Whitney sum decomposition
    \begin{equation}
    \label{new-eq-1}
    S_{I} \times V_{I} = \mathcal{T} S_{I} \oplus \mathcal{R}_{I},
    \end{equation}
    where $\mathcal{T} S_{I}$ stands for the tangent bundle of $S_{I}$, with the following properties:
    \begin{enumerate}
        \item[\textup{(i)}]
        The fiber $\mathcal{R}_{I}(x)$ of $\mathcal{R}_{I}$ over $x \in S_{I}$ can be written as $\spanned\{r(x)\}$, where $r \colon S_{I} \to V_{I}$ is continuous, with $\norm{r(x)} = 1$ for all $x \in S_{I}$, and $r(x) \in \dot{K}_{I}$ if $x \in \dot{S}_{I}$;
        \item[\textup{(ii)}]
        For any $x \in S_{I}$ the tangent space $\mathcal{T}_{x} S_{I}$ of $S_{I}$ at $x$ intersects $K_{I}$ only at $0$;
\usetagform{default}
        \item[\textup{(iii)}]
        There are $C > 0$ and $\nu > 0$ such that
        \begin{equation}
        \label{eq:exp-sep}
        \frac{\norm{DP^n(x) r(x)}}{\norm{DP^n(x) w}} \le C e^{{-\nu} n}
        \end{equation}
        for all $x \in S_{I}$, non-zero $w \in \mathcal{T}_{x}S_{I}$ and all $n \in \NN$.
    \end{enumerate}
\end{proposition}
\begin{proof}
In view of Remark~\ref{rem-1}, this is \cite[Thm.~5.1]{J-M-W} applied to the restriction of $P$ to $C_{I}$.
\end{proof}

We need to show that (H$_{k-1}$) implies (H$_{k}$).  Again using Remark~\ref{rem-1}, we shall prove only the last inductive step, that is, that (H$_{N-2}$) implies (H$_{N-1}$).  So, we need to show that for any ergodic invariant measure supported on $\dot{S}_{I}$ with $\card{I} = N - 1$ its principal Lyapunov exponent is smaller than its (unique) external Lyapunov exponent.  For notational simplicity we take $I = \{1, \dots, N - 1\}$.

\subsection{Geometrical considerations}
For $x \in S$ we define
\begin{equation*}
\mathcal{C}_1(x) := \{\, v \in V: \exists\ (x^{(k)})_{k = 1}^{\infty} \subset S \setminus \{x\}, x^{(k)} \to x, \frac{x^{(k)} - x}{\norm{x^{(k)} - x}} \to v \,\}.
\end{equation*}
$\mathcal{C}(x) : = \{\, {\lambda} x : \lambda \ge 0, v \in  \mathcal{C}_1(x)\,\}$ is called the {\em tangent cone\/} of $S$ at $x$.  $\mathcal{C}(x)$ is a non-trivial (that is, not containing only $0$) closed subset of $V$.

\medskip
Observe that if $x \in S^{\circ}$ then it follows from Theorem~\ref{thm0}(i) that the orthogonal projection of $\mathcal{C}(x)$ along $(1, \dots, 1)$ equals $\{\, v \in V : \sum_{i=1}^{N} v_i = 0 \,\}$.  When $S^{\circ}$ is $C^1$, $\mathcal{C}(x) = \mathcal{T}_{x} S$, the tangent space of $S$ at $x$, for all $x \in S^{\circ}$.

For a carrying simplex neatly embedded in $C$, at each $x \in \partial S$ the tangent space $\mathcal{T}_{x} S$ equals $\spanned{\mathcal{C}(x)}$.  To illustrate how, generally, the families of tangent cones can look like, let us go back to the example in Figure~\ref{fig:1}.  For $x \in S_{\{1,2\}} \setminus \{(0,1,0)\}$ there holds $\spanned{\mathcal{C}(x)} = V_{\{1,2\}}$, whereas $\spanned{\mathcal{C}((0,1,0))}$ is a two-dimensional subspace transverse to $V_{\{2\}}$ (for proofs see~\cite[Example 8.3]{J-M-W}).
\smallskip

Further, $DP(x) \mathcal{C}(x) = \mathcal{C}(P(x))$ for any $x \in S$.

\medskip
Families of tangent cones will play a significant role in the proof of the Main Theorem.  Assuming, {\em per contra\/}, that for some ergodic invariant measure the external Lyapunov exponent is not larger than the principal Lyapunov exponent we will construct an invariant measurable family of tangent vectors, and it will turn out that its existence is incompatible with the inequalities on the Lyapunov exponents.

\medskip
Let $e_N = (0, \dots, 0, 1)$.  For $x \in S_{I}$ we decompose $z \in \mathcal{C}_1(x)$ as
\begin{equation}
\label{decomp_z}
z = \alpha(z) e_N - \beta(z) r(x) + w(z),
\end{equation}
with real $\alpha(z)$, $\beta(z)$ and $w(z) \in \mathcal{T}_{x} S_I$.  Since the $N$-th coordinates of $r(x)$ and $w(z)$ are zero, the $N$-th coordinate of $z$ is non-negative, and $\norm{z} = 1$, we have $0 \le \alpha(z) \le 1$.

\begin{lemma}
\label{lm-1}
$\beta(z) \ge 0$ for any $z \in \mathcal{C}_1(x)$, $x \in \dot{S}_{I}$.
\end{lemma}
\begin{proof}
Suppose that there are $x \in \dot{S}_I$ and $z \in \mathcal{C}_1(x)$ such that $\beta(z)$ in~\eqref{decomp_z} is negative.  Then it follows from Proposition~\ref{prop:induction}(i) that $\alpha(z) e_N - \beta(z) r(x) \in K^{\circ}$.

Assume that $w(z) \ne 0$.  Since $z \in \mathcal{C}_1(x)$, there is a sequence $(x^{(k)})_{k=1}^{\infty} \subset S \setminus \{x\}$ converging to $x$ and such that for each $\epsilon > 0$
\begin{equation}
\label{eq-1}
\left\lVert \frac{x^{(k)} - x}{\norm{x^{(k)} - x}} - (\alpha(z) e_{N} - \beta(z) r(x) + w(z) ) \right\rVert < \epsilon
\end{equation}
for $n$ sufficiently large.  By Proposition~\ref{prop:induction}, $\dot{S}_I$ is a $C^1$ $(N - 2)$-dimensional manifold, so there exists a $C^1$ arc $A \subset S$ tangent at $x$ to $w(z)$. Consequently, for any (sufficiently large) $k$ there exists $\tilde{x}^{(k)} \in A$ such that $\norm{\tilde{x}^{(k)} - x} = \norm{w(z)} \, \norm{x^{(k)} - x}$.  Further, as $A$ is tangent at $x$ to $\gamma(z) w$, for each $\epsilon > 0$ there holds
\begin{equation*}
\left\lVert \frac{\tilde{x}^{(k)} - x}{\norm{\tilde{x}^{(k)} - x}} - \frac{w(z)}{\norm{w(z)}} \right\rVert < \frac{\epsilon}{\norm{w(z)}},
\end{equation*}
consequently
\begin{equation}
\label{eq-2}
\left\lVert \frac{\tilde{x}^{(k)} - x}{\norm{x^{(k)} - x}} - w(z)  \right\rVert < \epsilon,
\end{equation}
for $k$ sufficiently large.  Putting together \eqref{eq-1} and~\eqref{eq-2} we see that
\begin{equation*}
\left\lVert \frac{x^{(k)} - \tilde{x}^{(k)}}{\norm{x^{(k)} - x}} - (\alpha(z) e_N - \beta(z) r(x)) \right\rVert < 2 \epsilon
\end{equation*}
for $k$ sufficiently large.  Take now $\epsilon > 0$ so small that vectors within $2\epsilon$ of $\alpha(z) e_N - \beta(z) r(x)$ belong to $K^{\circ}$.  Therefore, for some $k$, $\tilde{x}^{(k)} \ll x^{(n)}$, which is impossible.  The case $w(z) = 0$ is considered in a similar (but simpler) way.
\end{proof}

\begin{lemma}
\label{lm-2}
$\beta(z)/\alpha(z)$ is positive and bounded away from zero, uniformly in $z \in \mathcal{C}_1(x) \setminus V_I$, $x \in \dot{S}_I$.
\end{lemma}
\begin{proof}
For $x \in S_I$ we write
\begin{equation*}
DP^{-2}(x) e_N = b(x) e_N + c(x) r(P^{-2}x) + w(x),
\end{equation*}
where $b(x)$ and $c(x)$ are reals, and $w(x) \in \mathcal{T}_{P^{-2}x}S_I$.

For any $x \in S_{I}$ there is a nonempty $J \subset I$ such that $x \in \dot{S}_J$.  It follows from (H3$'$) that $DP^{-2}(x) e_N \in \dot{K}_J$.   As, by Proposition~\ref{prop:induction}(i) with $k = \card{J}$, $r(P^{-2}x) \in \dot{K}_J$, there holds $c(x) > 0$.  Since the bundle decomposition $S_{I} \times V = \mathcal{R}_I \oplus \mathcal{T}S_I \oplus (S_I \times V_{\{N\}})$ is continuous, the function $c$ is continuous, too.  Hence $c(x)$ is bounded away from zero, uniformly in $x \in S_{I}$.

Take $x \in \dot{S}_I$ and $z \in \mathcal{C}_1(x) \setminus V_{I}$.  It follows from~\eqref{decomp_z} that
\begin{multline*}
DP^{-2}(x) z
\\
= {\alpha(z)} b(x) e_N + ({\alpha(z)} c(x) - {\beta(z)} a(x)) r(P^{-2}x)  + (DP^{-2} w(z) + \tilde{w}),
\end{multline*}
where $a(x)$ denotes the norm of $DP^{-2}(x)|_{\mathcal{R}_x}$, and $\tilde{w} \in \mathcal{T}_{P^{-2}x}S_I$.  Applying Lemma~\ref{lm-1} to $DP^{-2}(x) z/\norm{DP^{-2}(x) z} \in \mathcal{C}_1(P^{-2}x)$ we obtain
\begin{equation*}
\frac{\beta(z)}{\alpha(z)} \ge \frac{c(x)}{a(x)}.
\end{equation*}
Since $a(x)$ is bounded uniformly in $x \in S_I$, the result follows.
\end{proof}

\begin{figure}[H]
  \centering
  \begin{tikzpicture}[scale=1]
    \draw[->] (0,0) -- (4,0) node[right] {$r(x)$};

    \filldraw (0,0) circle (2pt) node[below] {$0$};

    \draw[blue, thick] (0,0) -- (-4,0) ;

    \draw[->] (0, 0) -- (0,4) node[above] {$e_n$};

    \draw[blue, thick] (0,0) -- (-1,4);

    \begin{scope}[on background layer]
        \fill[blue, opacity=.2] (-4,4) -- (-1, 4) -- (0, 0) -- (-4,0);
    \end{scope}

  \end{tikzpicture}
    \caption{\small The picture plane is the orthogonal projection of $V$ on $\spanned\{r(x), e_N\}$.  A geometric interpretation of Lemmas~\ref{lm-1} and~\ref{lm-2} is that the tangent cone $\mathcal{C}(x)$ is contained in the blue-filled domain, and that domain is independent of $x$.}
    \label{fig-2}
\end{figure}
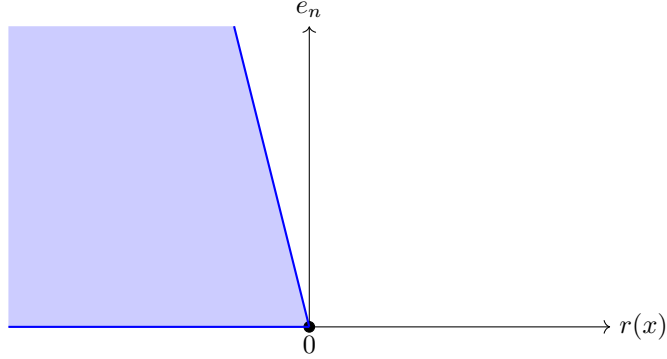

Observe that in Lemmas~\ref{lm-1} and~\ref{lm-2} we do not assume $S$ to be convex.

\begin{lemma}
\label{lm-3}
$\alpha(z)/\beta(z)$ is positive and bounded away from zero, uniformly in $z \in \mathcal{C}_1(x) \setminus V_I$, $x \in \dot{S}_I$.
\end{lemma}
\begin{proof}
Take $x \in \dot{S}_I$.  Because $R(S_I) = \triangle \cap C_I$ and $R(x) \in \triangle \cap \dot{C}_I$, we can find a simplex, $\conv\{y^{(1)}, \dots, y^{(N-1)}\} \subset \triangle \cap \dot{C}_I$, having $R(x)$ as its barycentre.  Consider the hyperplane $L$ passing through $u_N = (R|_{S})^{-1}(0, \dots, 0, 1)$, $x^{(1)} = (R|_{S})^{-1} (y^{(1)})$, \dots, $x^{(N - 1)} = (R|_{S})^{-1} (y^{(N - 1)})$.  A vector $p$ normal to $L$ can be chosen to have all coordinates positive, so the intersection $L \cap C$ divides $C$ into two sets, a bounded one, $L^{-}$, containing the origin, and an unbounded one, $L^{+}$.  By the convexity assumption, the image $(R|_{S})^{-1}(\conv\{{(0, \dots, 0, 1)}, y^{(1)}, \dots, y^{(N - 1)}\})$ is contained in $L \cup L^{+}$.  Observe that the image of $\{\, t_{0}(0, \dots, 0, 1) + t_{1}y^{(1)} + \ldots + {} t_{N - 1}y^{(N - 1)}: t_0 + t_1 + \ldots + {} t_{N - 1} = 1, \ t_0 \ge 0, \ t_1, \dots, t_{N - 1} > 0 \,\}$ under $(R|_{S})^{-1}$ is a neighbourhood of $x$ in the relative topology of $S$.

The above construction can be repeated when we replace $\conv\{y^{(1)}, \dots, \allowbreak y^{(N-1)}\}$ by its image under the homothety with centre $R(x)$ and ratio $\epsilon \in (0, 1]$.  Let $\epsilon \to 0^{+}$.  Then the hyperplanes $L$ converge to the hyperplane $\widetilde{L}$ containing $x + \mathcal{T}_{x} S_{I}$ and passing through $u_N$.  Any $z \in \mathcal{C}_1(x)$, considered a bound vector with initial point at $x$, has its terminal point in $\widetilde{L} \cup \widetilde{L}^{+}$.

A normal vector to $\widetilde{L}$ can be chosen to belong to $K^{\circ}$ (denote such a normalized vector by $\tilde{p}(x)$).  From the previous paragraph it follows that for any $z \in \mathcal{C}_1(x)$ there holds $\langle z, \tilde{p}(x) \rangle \ge 0$, consequently, taking~\eqref{decomp_z} into account we obtain
\begin{equation*}
\alpha(z) \langle e_N, \tilde{p}(x) \rangle - \beta(z) \langle r(x), \tilde{p}(x) \rangle + \gamma(z) \langle w(z), \tilde{p}(x) \rangle \ge 0.
\end{equation*}
As $\alpha(z) > 0$, $\langle e_N, \tilde{p}(x) \rangle > 0$, $\langle r(x), \tilde{p}(x) \rangle > 0$ and $\langle w(z), \tilde{p}(x) \rangle = 0$, we have that
\begin{equation*}
\frac{\alpha(z)}{\beta(z)} \ge \frac{\langle r(x), \tilde{p}(x) \rangle}{\langle e_N, \tilde{p}(x) \rangle}.
\end{equation*}
$r(x)$ depends continuously on $x \in S_I$, and $\tilde{p}$ can be extended to a continuous function on the whole of $S_I$, satisfying the inequalities $\langle r(x), \tilde{p}(x) \rangle > 0$ and $\langle e_N, \tilde{p}(x) \rangle > 0$.  The conclusion of the lemma thus follows.
\end{proof}

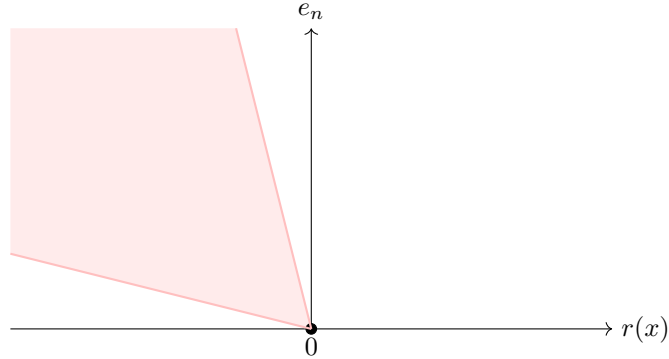
\begin{figure}[H]
  \centering
  \begin{tikzpicture}[scale=1]
    \draw[->] (0,0) -- (4,0) node[right] {$r(x)$};

    \draw (0,0) -- (-4,0) ;
    \draw[->] (0, 0) -- (0,4) node[above] {$e_n$};
    \filldraw (0,0) circle (2pt) node[below] {$0$};
    \draw[pink, thick] (0,0) -- (-4,1);
    \draw[pink, thick] (0,0) -- (-1,4);
    \begin{scope}[on background layer]
        \fill[pink, opacity=.3] (-4,4) -- (-1, 4) -- (0, 0) -- (-4,1) -- (-4,0);
    \end{scope}

  \end{tikzpicture}
    \caption{\small As in Figure~\ref{fig-2}, the picture plane is the orthogonal projection of $V$ on $\spanned\{r(x), e_N\}$.  A geometric interpretation of Lemmas~\ref{lm-1}, \ref{lm-2} and~\ref{lm-3} is that $\mathcal{C}(x) \setminus V_I$ is contained in the pink-filled domain, and that domain is independent of $x$.}
    \label{fig-3}
\end{figure}

\bigskip
For $\epsilon > 0$ sufficiently small consider the set
\begin{multline*}
B_{\epsilon} := \{\, (x , v) \in S_I \times V : \dist(x ,\partial S_I) \ge \epsilon \\
\text{ and } v = {\alpha} e_N + w \text{ such that } \norm{v} = 1,
w \in \mathcal{T}_{x}S_I \text{ and } \alpha \ge \epsilon \norm{w} \, \}.
\end{multline*}
It is easy to see that $B_{\epsilon}$ is closed in $S_I \times V$, hence compact.

We can (and do) take $\eta_{\epsilon} > 0$ such that for any $x \in \dot{S}_I$ with $\dist(x ,\partial S_I) \ge \epsilon$ there holds $x - {\eta_{\epsilon}} r(x) \in \dot{C}_I$.  One has then, by Theorem~\ref{thm0}(ii), that $x - {\eta_{\epsilon}} r(x) \in \Gamma$ and $x + {\eta_{\epsilon}} r(x) \in C \setminus \Gamma$ for any $x \in \dot{S}_I$ with $\dist(x ,\partial S_I) \ge \epsilon$.

For $(x, v) \in B_{\epsilon}$ take the set $x + [0, {\delta}_{\epsilon}] v + [{- \eta_{\epsilon}}, {\eta_{\epsilon}}] r(x)$, where $\delta_{\epsilon} > 0$, independent of $(x, v)$, will be set soon.

First, we take $\delta_{\epsilon} > 0$ so small that $x + {\delta_{\epsilon}} v - {\eta_{\epsilon}} r(x) \in C$ and $x + {\delta_{\epsilon}} v + {\eta_{\epsilon}} r(x)\in C$ for all $(x, v) \in B_{\epsilon}$.  By the convexity of $C$, $x + {\delta} v + {\eta} r(x) \in C$ for all $(x, v) \in B_{\epsilon}$,  $\delta \in [0, \delta_{\epsilon}]$ and $\eta \in [- \eta_{\epsilon}, \eta_{\epsilon}]$.

\begin{lemma}
  \label{add-lemma-1}
  $\delta_{\epsilon} > 0$ can be taken so small that $x + {\delta} v - {\eta_{\epsilon}} r(x) \in \Gamma \setminus S$ and $x + {\delta} v + {\eta_{\epsilon}} r(x) \in C \setminus \Gamma$, for all $\delta \in [0, \delta_{\epsilon}]$.
\end{lemma}
\begin{proof}
  Suppose to the contrary that there is no such $\delta_{\epsilon}$.  It follows that there are sequences $(x^{(k)})_{k = 1}^{\infty}$, $(v^{(k)})_{k = 1}^{\infty}$, $(x^{(k)}, v^{(k)}) \in B_{\epsilon}$, and $\delta^{(k)} \to 0$ as $k \to \infty$, such that $x^{(k)} + \delta^{(k)} v^{(k)} - {\eta_{\epsilon}} r(x^{(k)}) \notin \Gamma \setminus S$ [$x^{(k)} + \delta^{(k)} v^{(k)} + {\eta_{\epsilon}} r(x^{(k)}) \notin C \setminus \Gamma$].  By compactness, one can extract subsequences (denoted as before) such that $x^{(k)} + \delta^{(k)} v^{(k)}- {\eta_{\epsilon}} r(x^{(k)})$ converge to some $x  - {\eta_{\epsilon}} r(x) \in \dot{C}_I$ [$x^{(k)} + \delta^{(k)} v^{(k)} + {\eta_{\epsilon}} r(x_n)$ converge to some $x + {\eta_{\epsilon}} r(x) \in \dot{C}_I$].  It follows from Theorem~\ref{thm0} that $x - {\eta_{\epsilon}} r(x) \in \Gamma \setminus S$ [$x + {\eta_{\epsilon}} r(x) \in \Inte_{C}(C \setminus \Gamma)$].  But $\Gamma \setminus S$ [$\Inte_{C}(C \setminus \Gamma)$] is open in the relative topology of $C$, a contradiction.
\end{proof}

Fix, for a moment, $(x, v) \in B_{\epsilon}$.  that there is $j_{(x, v)}(\delta) \in [- \eta_{\epsilon}, \eta_{\epsilon}]$ such that
\begin{equation*}
x + {\delta} v - j_{(x, v)}(\delta) \, r(x) \in S.
\end{equation*}
If follows from Theorem~\ref{thm0} that for a fixed $\delta$ such a $j_{(x, v)}(\delta)$ is unique (otherwise, for $\delta = 0$ we would find two points in $\dot{S}_I$ in the $\ll_{I}$ relation, and for $\delta \in (0, \eta]$ we would find two points in $S^{\circ}$ in the $<$ relation).  We have
\begin{equation}
  \label{add-eq-1}
  S \cap (x + [0, {\delta}_{\epsilon}] v + [{- \eta_{\epsilon}}, {\eta_{\epsilon}}] r(x)) = x + \{\, {\delta} v - j_{(x, v)}(\delta) r(x) : \delta \in [0, {\delta}_{\epsilon}] \,\}.
\end{equation}

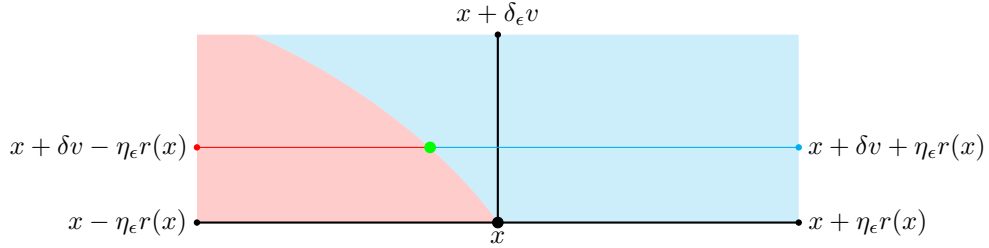
\begin{figure}[H]
  \centering
  \begin{tikzpicture}[scale=1]

    \draw[thick] (0,0) -- (4,0) node[right] {$x + {\eta_{\epsilon}} r(x)$};

    \filldraw (0,0) circle (2pt) node[below] {$x$};

    \draw[thick] (0,0) -- (-4,0) node[left] {$x - {\eta_{\epsilon}} r(x)$};

    \draw[thick] (0, 0) -- (0,2.5) node[above] {$x + {\delta_{\epsilon}}v$};

    \draw[cyan] (-.9, 1)  -- (4, 1);

    \draw[red] (-4, 1) -- (-.9,1);

    \filldraw[green] (-.9,1) circle (2pt);
    \filldraw (4,0) circle (1pt);
    \filldraw (-4,0) circle (1pt);
    \filldraw (0,2.5) circle (1pt);
    \filldraw[red] (-4,1) circle (1pt);
    \filldraw[cyan] (4,1) circle (1pt);

    \begin{scope}[on background layer]
        \fill[red, opacity=.2] (0, 0) .. controls (-0.125/8, 0.125/4) and (-1,1.5) .. (-3.25,2.5) -- (-3.25,2.5) .. controls (-4,2.5) .. (-4,2.5) -- (-4,2.5) .. controls (-4,0) .. (-4,0);
        \fill[cyan, opacity=.2] (0, 0) .. controls (-0.125/8, 0.125/4) and (-1,1.5) .. (-3.25,2.5) -- (-3.25,2.5) .. controls (4,2.5) .. (4,2.5) -- (4,4) .. controls (4,0) .. (4,0) -- (4,0) .. controls (0,0) .. (0,0);
    \end{scope}
    \node[left] at (-4, 1) {$x + {\delta}v - {\eta_{\epsilon}} r(x)$};
    \node[right] at (4, 1) {$x + {\delta}v + {\eta_{\epsilon}} r(x)$};

  \end{tikzpicture}
    \caption{\small The picture plane is $x + \spanned\{v, r(x)\}$. The cyan-filled domain is $(C \setminus \Gamma) \cap (x + \spanned\{v, r(x)\})$, the red-filled domain is $(\Gamma \setminus S) \cap (x + \spanned\{v, r(x)\})$.  The green-filled circle represents the intersection of $S$ with $x + {\delta} v + [- {\eta_{\epsilon}}, {\eta_{\epsilon}}] r(x)$, for some $\delta \in (0, \delta_{\epsilon})$.}
    \label{fig-4}
\end{figure}

\begin{lemma}
  \label{add-lemma-2}
  The assignment
  \begin{equation*}
    B_{\epsilon} \times [0, \delta_{\epsilon}] \ni (x, v, \delta) \mapsto j_{(x, v)}(\delta) \in [- \eta_{\epsilon}, \eta_{\epsilon}]
  \end{equation*}
  is continuous.
\end{lemma}
\begin{proof}
  Suppose not.  This means that there is a sequence $(x^{(k)}, v^{(k)}, \delta^{(k)})_{k=1}^{\infty}$ converging to $(x, v, \delta)$ such that $j_{(x^{(k)}, v^{(k)})}(\delta^{(k)})$ does not converge to $j_{(x, v)}(\delta)$.  By compactness, we can extract a subsequence (denoted as above) such that $j_{(x^{(k)}, v^{(k)})}(\delta^{(k)})$ converge to $\tilde{j} \ne j_{(x, v)}(\delta)$.  But $x^{(k)} + \delta^{(k)} v^{(k)} - j_{(x^{(k)}, v^{(k)})}(\delta) \allowbreak r(x^{(k)})$ belong to the closed set $S$, so their limit, $x + {\delta} v- \tilde{j} \, r(x)$ belongs to $S$, too.
\end{proof}

\begin{lemma}
  \label{add-lemma-3}
  For $(x, v) \in B_{\epsilon}$ fixed, the function $j_{(x, v)} \colon  [0, \delta_{\epsilon}] \to \RR$ is convex.
\end{lemma}
\begin{proof}
  The convex set
  \begin{equation*}
    \Gamma \cap (x + [0, {\delta}_{\epsilon}] v + [{- \eta_{\epsilon}}, {\eta_{\epsilon}}] r(x)) = x + \{\, {\delta} v - [j_{(x, v)}(\delta), {\eta_{\epsilon}}] r(x) : \delta \in [0, {\delta}_{\epsilon}] \,\}
  \end{equation*}
  corresponds to the epigraph of $j_{(x, v)}$.
\end{proof}

\medskip
Fix an ergodic invariant measure $\mu$ such that $\mu(\dot{S}_I) = 1$.  Denote by $\lambdaprinc(\mu)$ the principal Lyapunov exponent for $\mu$, and by $\lambdaext(\mu)$ the (unique) external Lyapunov exponent for $\mu$.  Let $O$ stand for the set of those $x \in \dot{S}_I$ for which the Oseledets splitting is defined.  $O$ is invariant, with $\mu(O) = 1$.

It follows from Proposition~\ref{prop:induction}(iii) that the smallest Lyapunov exponent on $S_{I} \times V_{I}$, that is, $\lambdaprinc(\mu)$, has multiplicity (in $S_{I} \times V_I$) one, and that the remaining Lyapunov exponents on $S_{I} \times V_I$, that is, those corresponding to $\mathcal{T}S_{I}$, are larger than $\lambdaprinc(\mu)$.  Recall that our purpose is to prove that $\lambdaext(\mu) > \lambdaprinc(\mu)$.  Suppose to the contrary that $\lambdaext(\mu) \le \lambdaprinc(\mu)$.

\smallskip
For any $x \in O$ we define a two-dimensional subspace $\mathcal{E}(x)$.
\begin{itemize}
  \item
  If $\lambdaext(\mu) < \lambdaprinc(\mu)$, let $\mathcal{E}(x)$ stand for the sum of the one-dimensional Oseledets subspace corresponding to the smallest Lyapunov exponent (that is, to $\lambdaext(\mu)$) and the one-dimensional Oseledets subspace corresponding to the second smallest Lyapunov exponent (that is, to $\lambdaprinc(\mu)$).
  \item
  If $\lambdaext(\mu) = \lambdaprinc(\mu)$, let $\mathcal{E}(x)$ stand for the two-dimensional Oseledets subspace corresponding to the smallest Lyapunov exponent.
\end{itemize}

By~\eqref{new-eq-1}, for each $x \in O$ there holds $V = \mathcal{E}(x) \oplus \mathcal{T}_{x} S_{I}$.  Furthermore, $\mathcal{R}_I(x) \subset \mathcal{E}(x)$ for each $x \in O$.

\smallskip
By \cite[Cor.~7.3]{Lian-Lu}, we can find a measurable family $\{\tilde{e}(x)\}_{x \in O}$ of unit vectors such that $(\tilde{e}(x), r(x))$ is a basis of $\mathcal{E}(x)$.  The vectors $\tilde{e}(x)$ can be chosen so that their $N$-th coordinates are positive.

For $k \in \NN$ denote by $O_k$ the set of all those $x \in O$ for which $(x, \tilde{e}(x)) \in B_{1/k}$.  The set $O_k$ is measurable.

Let a positive integer $k$ be fixed. For any $x \in O_k$ we define
\begin{equation*}
j_x(\delta) := j_{(x, \tilde{e}(x))}(\delta), \quad \delta \in [0, \delta_{1/k}].
\end{equation*}
It follows from Lemma~\ref{add-lemma-3} that $j'_{x}(0)$, the right derivative at $0$, exists.  Consequently, by~\eqref{add-eq-1}, $\mathcal{C}_1(x) \cap \mathcal{E}(x) = \{z(x)\}$, where
\begin{equation*}
  z(x) := \frac{\tilde{e}(x) - j'_{x}(0) r(x)}{\norm{\tilde{e}(x) - j'_{x}(0) r(x)}}.
\end{equation*}

\begin{figure}[H]
  \centering
  \begin{tikzpicture}[scale=1]
    \pgfmathsetmacro{\vectorlength}{veclen(-1,2)};

    \draw[thick] (0,0) -- (4,0) node[right] {$x + \eta_{1/k} r(x)$};

    \filldraw (0,0) circle (2pt) node[below] {$x$};

    \draw[thick] (0,0) -- (-4,0) node[left] {$x - \eta_{1/k} r(x)$};

    \draw[thick] (0, 0) -- (0,4) node[above] {$x + \delta_{1/k} \tilde{e}(x)$};

    \draw[green, thick] (0, 0) .. controls (-0.125/8, 0.125/4) and (-1.5,3.5) .. (-3.25,4);

    \draw[->,blue, thick] (0,0) -- (-4/\vectorlength, 8/\vectorlength) node[right=2pt] {$z(x)$};

    \begin{scope}[on background layer]
        \fill[red, opacity=.2] (-4,0) .. controls (-4,4) .. (-4,4) -- (-4, 4) .. controls (-3.25,4) .. (-3.25,4) -- (0, 0) .. controls (-0.125/8, 0.125/4) and (-1.5,3.5) .. (-3.25,4) -- (0,0) .. controls (-4,0) .. (-4,0);
        \fill[cyan, opacity=.2] (0, 0) .. controls (-0.125/8, 0.125/4) and (-1.5,3.5) .. (-3.25,4) -- (-3.25,4) .. controls (4,4) .. (4,4) -- (4,4) .. controls (4,0) .. (4,0) -- (4,0) .. controls (0,0) .. (0,0);
    \end{scope}
    \filldraw (-4,0) circle (1pt);
    \filldraw (4,0) circle (1pt);
    \filldraw (0,4) circle (1pt);

  \end{tikzpicture}
    \caption{\small The picture plane is $x + \mathcal{E}(x)$.  The cyan-filled domain is $(C \setminus \Gamma) \cap (x + \mathcal{E}(x))$, the red-filled domain is $(\Gamma \setminus S) \cap (x + \mathcal{E}(x))$, the green curve is $S \cap (x + \mathcal{E}(x))$. }
    \label{fig-5}
\end{figure}
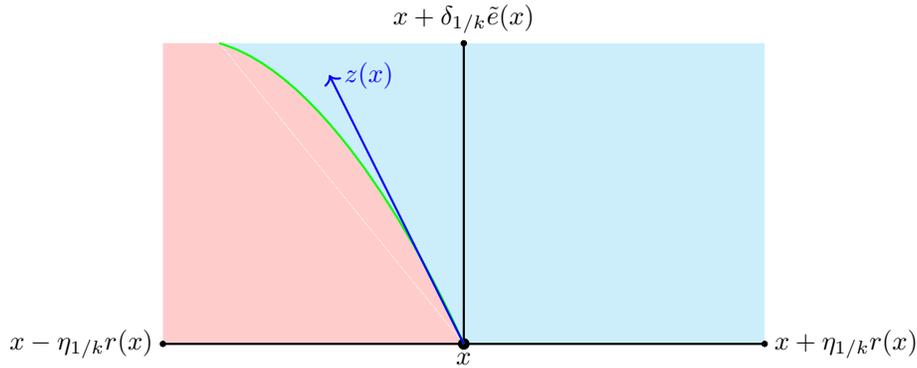

\begin{proposition}
  \label{add-prop1}
  \begin{enumerate}
    \item[\textup{(1)}]
        $j'_x(0)$ is positive, for all $x \in O_k$.
    \item[\textup{(2)}]
        The assignment
            \begin{equation*}
                O_k \ni x \mapsto j'_{x}(0) \in (0, \infty)
            \end{equation*}
       is $(\mathcal{B}(\dot{S}_{I}), \mathcal{B}(\RR))$-measurable.
  \end{enumerate}
\end{proposition}
\begin{proof}
(1) follows by Lemma~\ref{lm-2}.

(2) is a consequence of the fact that, by~\cite[Thm.~4.55]{AliB}, the assignment
\begin{equation*}
O_k \ni x \mapsto [\,  \delta \to j_x(\delta)] \in C([0,\delta_{1/k}], \RR)
\end{equation*}
is $(\mathcal{B}(\dot{S}_{I}), \mathcal{B}(C([0,\delta_{1/k}], \RR))$-measurable.
\end{proof}

We have thus obtained a measurable family $\{z(x)\}_{x \in O}$ of vectors in $\mathcal{C}_1(x) \cap \mathcal{E}(x)$ such that $(z(x), r(x))$ forms a basis of $\mathcal{E}(x)$.  Observe that $DP(x) z(x) \in \mathcal{C}_1(Px) \cap \mathcal{E}(Px)$ and its $N$-th coordinate is positive, so it is equal to $z(Px)$ multiplied by a positive scalar.

\subsection{Ergodic-theoretic considerations}
As in the proof of Lemma~\ref{lm-2}, let $DP^{-2}(x) r(x) = a(x) r(P^{-2}x)$.  Further, for any $x \in O$ define $d(x) > 0$ by $DP^{-2}(x) z(x) = d(x) z(P^{-2}x)$.  $w$ will stand for a generic tangent vector at some point in $S_I$.

Since $a(x) \cdot \ldots \cdot a(P^{-n+1}x) = \norm{DP^{-2n + 2}(x) r(x)}$, one has
\begin{equation}
\label{ineq-1}
\lim\limits_{n \to \infty} \frac{1}{n} \sum\limits_{i = 0}^{n - 1} \ln{a(P^{-2i}x)} = - 2 \lambdaprinc(\mu).
\end{equation}
We claim that
\begin{equation}
\label{ineq-2}
\lim\limits_{n \to \infty} \frac{1}{n} \sum\limits_{i = 0}^{n - 1} \ln{d(P^{-2i}x)} = - 2 \lambdaext(\mu).
\end{equation}
Recall that we suppose that $\lambdaext(\mu) \le \lambdaprinc(\mu)$.  If $\lambdaext(\mu) = \lambdaprinc(\mu)$, the claim is satisfied.  On the other hand, if $\lambdaext(\mu) < \lambdaprinc(\mu)$, then for any $x \in O$ and any $u \in \mathcal{E}(x) \setminus \mathcal{R}_I(x)$ there holds
\begin{equation*}
\lim\limits_{n \to \infty} \frac{1}{n} \sum\limits_{i = 0}^{n - 1} \ln{\norm{DP^{-i}(x)u}} = - \lambdaext(\mu).
\end{equation*}
Consequently, remembering that $d(x) \cdot \ldots \cdot d(P^{-n+1}x) = \norm{DP^{-2n + 2}(x) z(x)}$ we see that \eqref{ineq-2} holds.

We have thus obtained
\begin{equation}
\label{ineq-3}
\lim\limits_{n \to \infty} \frac{1}{2n} \sum\limits_{i = 0}^{n - 1} \frac{\ln{a(P^{-2i}x)}}{\ln{d(P^{-2i}x)}} = \lim\limits_{n \to \infty} \frac{1}{2n} \sum\limits_{i = 0}^{n - 1} \frac{\ln{a(P^{-2i-1}x)}}{\ln{d(P^{-2i-1}x)}} \le 0
\end{equation}
for all $x \in O$.

\smallskip
It follows from~\eqref{decomp_z} that
\begin{equation*}
DP^{-2}(x) z(x) = \alpha(z(x)) DP^{-2}(x) e_N - \beta(z(x)) DP^{-2}(x) r(x) + w.
\end{equation*}
Further, we can write
\begin{equation*}
DP^{-2}(x) e_N = b(x) e_N + c(x) r(P^{-2}x) + w.
\end{equation*}
Consequently,
\begin{multline*}
DP^{-2}(x) z(x)
\\
= \alpha(z(x)) b(x) e_N + \alpha(z(x)) c(x) r(P^{-2}x) - \beta(z(x)) a(x) r(P^{-2}x) + w.
\end{multline*}
On the other hand, by~\eqref{decomp_z},
\begin{equation*}
z(P^{-2}x) = \alpha(z(P^{-2}x)) e_N - \beta(z(P^{-2}x)) r(P^{-2}x) + w.
\end{equation*}
Therefore
\begin{multline*}
\alpha(z(x)) b(x) e_N + \alpha(z(x)) c(x) r(P^{-2}x) - \beta(z(x)) a(x) r(P^{-2}x) + w
\\
= \alpha(z(P^{-2}x)) d(x) e_N - \beta(z(P^{-2}x)) d(x) r(P^{-2}x) + w.
\end{multline*}
Since $e_N$ and $r(P^{-2}x)$ are linearly independent, one gets
\begin{equation*}
\beta(z(x)) a(x) = \alpha(z(x)) c(x)+ \beta(z(P^{-2}x)) d(x),
\end{equation*}
therefore
\begin{equation*}
\frac{a(x)}{d(x)} = \frac{\alpha(z(x)) c(x)}{\beta(z(x)) d(x)} + \frac{\beta(z(P^{-2}x))}{\beta(z(x))}.
\end{equation*}
By Lemma~\ref{lm-3}, the boundedness of $d(\cdot)$ and the boundedness away from zero of $c(\cdot)$, the first term on the right-hand side of the above equation is not less than some $D > 0$.  So we have
\begin{equation*}
\ln{\frac{a(x)}{d(x)}} \ge \ln\left(D + \frac{\beta(z(P^{-2}x))}{\beta(z(x))}\right) = : g(x), \quad x \in O.
\end{equation*}
Pick $M > 0$ such that $\mu(U) > 0$, where $U : = \{\, x \in O : \beta(z(P^{-2}x))/\beta(z(x)) < M \,\}$.  There holds
\begin{equation*}
\frac{D + \tfrac{\beta(z(P^{-2}x))}{\beta(z(x))}} {\tfrac{\beta(z(P^{-2}x))}{\beta(z(x))}} \ge 1 + \frac{D}{M} \quad \text{for } x \in U.
\end{equation*}
Put $\varrho : = \ln{(1 + \tfrac{D}{M})} > 0$.  We have
\begin{equation*}
g(x) \ge
\begin{cases}
\varrho + \ln{\beta(z(P^{-2}x))} - \ln{\beta(z(x))} & \text{for } x \in U
\\
\ln{\beta(z(P^{-2}x))} - \ln{\beta(z(x))}  & \text{for } x \in O \setminus U.
\end{cases}
\end{equation*}

\noindent Hence
\begin{align*}
\sum\limits_{i=0}^{n - 1} \ln{\frac{a(P^{-2i}x)}{d(P^{-2i}x)}} & \ge \sum\limits_{i=0}^{n - 1} g(P^{-2i}x)
\\
& {} \ge
\sum\limits_{i=0}^{n - 1} (\ln{\beta(z(P^{-2(i+1)}x))} - \ln{\beta(z(P^{-2i}x))}) + {\varrho} \sum\limits_{i=0}^{n - 1} \mathbf{1}_{U}(P^{-2i}x)
\\
& {} =
\ln{\beta(z(P^{-2n}x))} - \ln{\beta(z(x))} + {\varrho} \sum\limits_{i=0}^{n - 1} \mathbf{1}_{U}(P^{-2i}x),
\end{align*}
and, similarly
\begin{equation*}
\sum\limits_{i=0}^{n - 1} \ln{\frac{a(P^{-2i-1}x)}{d(P^{-2i-1}x)}} \ge \ln{\beta(z(P^{-2n-1}x))} - \ln{\beta(z(P^{-1}x))} + {\varrho} \sum\limits_{i=0}^{n - 1} \mathbf{1}_{U}(P^{-2i-1}x).
\end{equation*}
An application of the Birkhoff ergodic theorem gives that for $\mu$-a.e.~$x \in O$ there holds
\begin{equation*}
\lim\limits_{n \to \infty} \frac{1}{2n} \sum\limits_{i=0}^{n - 1} \mathbf{1}_{U}(P^{-2i}x) = \lim\limits_{n \to \infty} \frac{1}{2n} \sum\limits_{i=0}^{n - 1} \mathbf{1}_{U}(P^{-2i-1}x) = 2 \mu(U) > 0.
\end{equation*}
If $\beta$ is constant a.e.~on $O$ we have already obtained a contradiction to~\eqref{ineq-3}. So assume that $\beta $ is not constant.  Then there are $0 < r_1 < r_2$ such that the sets $O^{(1)} = \{ x \in O : \beta(z(x)) < r_1 \}$ and $O^{(2)} = \{ x \in O : \beta(z(x)) > r_2 \}$ have both positive measure.

By ergodicity, for $\mu$-a.e. $x \in O$ there is a positive integer $k$ such that $P^{-k} x \in O^{(1)}$ and a sequence $(n_m)_{m = 1}^{\infty}$, $n_m \to \infty$ as $m \to \infty$, such that $P^{-2n_m}x \in O^{(2)}$ for all $m$, or $P^{-2n_m-1}x \in O^{(2)}$ for all $m$.  For such a `good' $x$, write
\begin{multline*}
\ln{\beta(z(P^{-2n_m}x))} - \ln{\beta(z(x))}
\\
= \bigl(\ln{\beta(z(P^{-2n_m}x))} - \ln{\beta(z(P^{-k}x))}\bigr) + \bigl(\ln{\beta(z(P^{-k}x))} - \ln{\beta(z(x))}\bigr),
\end{multline*}
or
\begin{multline*}
\ln{\beta(z(P^{-2n_m-1}x))} - \ln{\beta(z(P^{-1}x))}
\\
= \bigl(\ln{\beta(z(P^{-2n_m-1}x))} - \ln{\beta(z(P^{-k}x))}\bigr) + \bigl(\ln{\beta(z(P^{-k}x))} - \ln{\beta(z(P^{-1}x))}\bigr).
\end{multline*}
The first terms in parentheses on the right-hand side are bounded below by $\ln{r_2} - \ln{r_1} > 0$, consequently
\begin{equation*}
\liminf\limits_{m \to \infty} \frac{\ln{\beta(z(P^{-2n_m}x))} - \ln{\beta(z(P^{-k}x))}}{2n_m} \ge 0,
\end{equation*}
or
\begin{equation*}
\liminf\limits_{m \to \infty} \frac{\ln{\beta(z(P^{-2n_m-1}x))} - \ln{\beta(z(P^{-k}x))}}{2n_m} \ge 0.
\end{equation*}
The second term, divided by $2n_m$, converges to $0$ as $m \to \infty$.

We have thus obtained that for $\mu$-a.e.~$x \in O$ there holds
\begin{equation}
\label{ineq-4}
\limsup\limits_{n \to \infty} \frac{1}{2n} \sum\limits_{i = 0}^{n - 1} \frac{\ln{a(P^{-2n}x)}}{\ln{d(P^{-2n}x)}} > 0 \quad \text{or} \quad \limsup\limits_{n \to \infty} \frac{1}{2n} \sum\limits_{i = 0}^{n - 1} \frac{\ln{a(P^{-2n-1}x)}}{\ln{d(P^{-2n-1}x)}} > 0,
\end{equation}
which contradicts~\eqref{ineq-3}.

\smallskip
The obtained contradiction concludes the proof of the Main Theorem.

\bigskip
\noindent {\em Acknowledgments.}  This work was supported by the NCN grant Maestro \newline 2013/08/A/ST1/00275.  Special thanks go to Antony Quas, to whom the last part of the proof of the Main Theorem is due.

I would like to warmly thank the referees, whose remarks have greatly contributed to improving the presentation.

\end{document}